\documentclass[12pt]{amsart}

\usepackage{amsmath,amssymb,amsthm}
\usepackage[all]{xy}
\usepackage[dvipdfm,colorlinks=true]{hyperref}
\usepackage{epic,eepic}
\usepackage{here}

\numberwithin{equation}{section}

\SelectTips{eu}{12}

\newtheorem{theorem}{Theorem}[section]
\newtheorem{proposition}[theorem]{Proposition}
\newtheorem{lemma}[theorem]{Lemma}
\newtheorem{corollary}[theorem]{Corollary}

\theoremstyle{definition}

\newtheorem{conjecture}[theorem]{Conjecture}

\theoremstyle{remark}
\newtheorem{remark}[theorem]{Remark}

\newcommand{\Z}{\mathbb{Z}}

\newcommand{\C}{\mathbb{C}}
\newcommand{\ZZ}{\mathcal{Z}}

\title{$p$-local stable splitting of quasitoric manifolds}

\author{Sho Hasui}
\address{Department of Mathematics, Kyoto University, Kyoto, 606-8502, Japan}
\email{s.hasui@math.kyoto-u.ac.jp}
\author{Daisuke Kishimoto}
\address{Department of Mathematics, Kyoto University, Kyoto, 606-8502, Japan}
\email{kishi@math.kyoto-u.ac.jp}
\author{Takashi Sato}
\address{Department of Mathematics, Kyoto University, Kyoto, 606-8502, Japan}
\email{t-sato@math.kyoto-u.ac.jp}
\date{\today}

\begin{document}

\maketitle

\baselineskip 18pt

\begin{abstract}
We show a homotopy decomposition of $p$-localized suspension $\Sigma M_{(p)}$ of a quasitoric manifold $M$ by constructing power maps. As an application we investigate the $p$-localized suspension of the projection $\pi$ from the moment-angle complex onto $M$, from which we deduce its triviality for $p>\dim M/2$. We also discuss non-triviality of $\pi_{(p)}$ and $\Sigma^\infty\pi$.
\end{abstract}

\section{Introduction and statement of results}

Manifolds which are now known as quasitoric manifolds were introduced by Davis and Januszkiewicz \cite{DJ} as a topological counterpart of smooth projective toric varieties, and have been the subject of recent interest in the study of manifolds with torus action. As well as toric varieties, quasitoric manifolds have been studied in a variety of contexts where combinatorics, geometry, and topology interact in a fruitful way. We refer the reader to the exposition \cite{BP} written by Buchstaber and Panov for basics of quasitoric manifolds. This note studies a topological aspect of quasitoric manifolds involving their $p$-localized suspension. A quasitoric manifold $M$ over a simple $n$-polytope $P$ is by definition a $2n$-manifold on which the compact $n$-torus $T^n$ acts in such a way that the orbit space $M/T^n$ is identified with the simple polytope $P$ as manifolds with corners. A fundamental example of quasitoric manifolds is the complex projective space $\C P^n$ which is the only quasitoric manifold over the $n$-simplex, whereas there are several quasitoric manifolds on the same simple polytope in general. Observe that since $\C P^n$ admits power maps, the $p$-localization of the suspension $\Sigma\C P^n_{(p)}$ splits into a wedge of $p-1$ spaces as in \cite{MNT}. We prove that any quasitoric manifold also admits power maps, and as a consequence the $p$-localization of its suspension splits into a wedge of $p-1$ spaces.

\begin{theorem}
\label{main-split}
For a quasitoric manifold $M$ there is a homotopy equivalence
$$\Sigma M_{(p)}\simeq X_1\vee\cdots\vee X_{p-1}$$
such that for each $i$, $\widetilde{H}_*(X_i;\Z)=0$ unless $*\equiv 2i+1\mod 2(p-1)$.
\end{theorem}

As a corollary we get a kind of rigidity of quasitoric manifolds over the same polytope, which also follows from a more general result Proposition \ref{split-even}.

\begin{corollary}
\label{main-rigid}
Let $M,N$ be quasitoric manifolds over the same simple $n$-polytope. For $p>n$ there is a homotopy equivalence
$$\Sigma M_{(p)}\simeq\Sigma N_{(p)}.$$
\end{corollary}

To a simplicial complex $K$ we can assign a space $\ZZ_K$ which is called the moment-angle complex for $K$ (see \cite{DJ,BP}). The fundamental construction involving quasitoric manifolds is that every quasitoric manifold over a simple polytope $P$ is obtained by the quotient of a certain free torus action on the moment-angle complex $\ZZ_{K(P)}$, where $K(P)$ denotes the boundary of the dual simplicial polytope of $P$. Then for a quasitoric manifold $M$ over $P$ the projection $\pi\colon\ZZ_{K(P)}\to M$ is of particular importance. We investigate the $p$-localization of the suspension of this projection through the $p$-local stable splitting of Theorem \ref{main-split}. Let $K$ be a simplicial complex on the vertex set $V$. Recall from \cite{BBCG} that there is a homotopy equivalence
\begin{equation}
\label{BBCG}
\Sigma\ZZ_K\simeq\bigvee_{\emptyset\ne I\subset V}\Sigma^{|I|+2}|K_I|
\end{equation}
where $K_I$ denotes the full subcomplex of $K$ on the vertex set $I\subset V$, i.e. $K_I=\{\sigma\in K\,\vert\,\sigma\subset I\}$, and $|K_I|$ means the geometric realization of $K_I$. We identify the map $\Sigma\pi_{(p)}\colon\Sigma(\ZZ_{K(P)})_{(p)}\to\Sigma M_{(p)}$ through the homotopy equivalences of Theorem \ref{main-split} and \eqref{BBCG}. Note that if $P$ has $m$ facets, then the vertex set of $K(P)$ is $[m]:=\{1,\ldots,m\}$.

\begin{theorem}
\label{main-projection}
Let $M$ be a quasitoric manifold over a simple polytope $P$ with $m$ facets. Then through the homotopy equivalences of Theorem \ref{main-split} and \eqref{BBCG}, the map $\Sigma\pi_{(p)}\colon\Sigma(\ZZ_{K(P)})_{(p)}\to\Sigma M_{(p)}$ is identified with a wedge of maps
$$\bigvee_{\substack{\emptyset\ne I\subset [m]\\|I|\equiv i\mod p-1}}(\Sigma^{|I|+2}|K(P)_I|)_{(p)}\to X_i.$$
for $i=1,\ldots,p-1$.
\end{theorem}

\begin{corollary}
\label{main-projection-trivial}
Let $M$ be a quasitoric manifold over a simple $n$-polytope $P$. For $p>n$, the map $\Sigma\pi_{(p)}\colon\Sigma(\ZZ_{K(P)})_{(p)}\to\Sigma M_{(p)}$ is null homotopic.
\end{corollary}

We also discuss necessity of suspension and localization for triviality of the projection $\pi\colon\ZZ_{K(P)}\to M$ in Corollary \ref{main-projection-trivial}. Consider the complex projective space $\C P^1$ as a quasitoric manifold. Then the projection $\pi$ is the Hopf map $S^3\to\C P^1$, so neither $\Sigma^\infty\pi$ nor $\pi_{(p)}$ for any $p$ are null homotopic. We will discuss this problem for more general quasitoric manifolds.

The authors are grateful to Kouyemon Iriye and Shuichi Tsukuda for useful comments.

\section{Cohomology of quasitoric manifolds}

This section collects basic properties of the cohomology of quasitoric manifolds which will be used later. Let $P$ be a simple $n$-polytope, and let $M$ be a quasitoric manifold over $P$. Put $f_i(P)$ to be the number of $(n-i-1)$-dimensional faces of $P$ for $i=-1,0,\ldots,n-1$. The $h$-vector of $P$ is defined by $(h_0(P),\ldots,h_n(P))$ such that for $k=0,\ldots,n$,
$$h_k(P)=\sum_{i=0}^k(-1)^{k-i}\binom{n-i}{n-k}f_{i-1}(P).$$
It is known that the module structure of the cohomology of $M$ is described by the $h$-vector of $P$, implying that the module structure depends only on $P$.

\begin{proposition}
[Davis and Januszkiewicz {\cite[Theorem 3.1]{DJ}} (cf. \cite{BP})]
\label{h-vector}
Let $M$ be a quasitoric manifold over $P$. Then we have 
$$H^{\rm odd}(M;\Z)=0\quad\text{and}\quad H^{2i}(M;\Z)\cong\Z^{h_i(P)}.$$
\end{proposition}

Let $K$ be a simplicial complex on the vertex set $[m]$. The moment-angle complex $\ZZ_K$ is defined by
$$\ZZ_K:=\bigcup_{\sigma\in K}D(\sigma)\quad(\subset(D^2)^m)$$
where $D(\sigma)=\{(x_1,\ldots,x_m)\in(D^2)^m\,\vert\,|x_i|=1\text{ whenever }i\not\in\sigma\}$ and $D^2$ is regarded as the unit disk of $\C$. Then the canonical action of $T^m$ on $(D^2)^m$ restricts to the action of $T^m$ on $\ZZ_K$. Let $M$ be a quasitoric manifold over a simple $n$-polytope $P$ with $m$ facets. Then we may regard the vertex set of $K(P)$ is $[m]$. As in \cite{DJ,BP}, $M$ is obtained by quotienting out the moment-angle complex $\ZZ_{K(P)}$ by a certain free $T^{m-n}$-action which is the restriction of the canonical $T^m$-action. Then there is a homotopy fibration 
\begin{equation}
\label{fibration}
\ZZ_{K(P)}\xrightarrow{\pi}M\xrightarrow{\alpha}BT^{m-n}.
\end{equation}
One easily sees that $\ZZ_{K(P)}$ is 2-connected (cf. \cite{BP}), hence the transgression $H^1(T^{m-n};\Z)\to H^2(M;\Z)$ associated with the fibration $T^{m-n}\to\ZZ_{K(P)}\xrightarrow{\pi}M$ is an isomorphism. In particular the induced map $\alpha^*\colon H^2(BT^{m-n};\Z)\to H^2(M;\Z)$ is an isomorphism. It is also known as in \cite[Theorem 4.14]{DJ} (cf. \cite{BP}) that the cohomology ring $H^*(M;\Z)$ is generated by 2-dimensional elements. We record these properties of the cohomology of $M$.

\begin{proposition}
\label{H^2}
Let $M$ be a quasitoric manifold over a simple $n$-polytope $P$ with $m$ facets. 
\begin{enumerate}
\item The transgression $H^1(T^{m-n};\Z)\to H^2(M;\Z)$ associated with the fibration $T^{m-n}\to\ZZ_{K(P)}\xrightarrow{\pi}M$ is an isomorphism.
\item The map $\alpha^*\colon H^2(BT^{m-n};\Z)\to H^2(M;\Z)$ is an isomorphism.
\item The cohomology ring $H^*(M;\Z)$ is generated by $H^2(M;\Z)$.
\end{enumerate}
\end{proposition}

\section{Proofs of the main results}

Let $P$ be a simple $n$-polytope with $m$ facets, and let $M$ be a quasitoric manifold over $P$. We construct power maps of $M$. Let $u$ be an integer. By the definition of moment-angle complexes, the degree $u$ self-map of $S^1$ induces a self-map $\underline{u}\colon\ZZ_{K(P)}\to\ZZ_{K(P)}$. 

\begin{lemma}
\label{power-map}
There is a self-map $\underline{u}\colon M\to M$ satisfying 
$$\underline{u}^*=u^k\colon H^{2k}(M;\Z)\to H^{2k}(M;\Z).$$
\end{lemma}

\begin{proof}
Since the self-map $\underline{u}\colon\ZZ_{K(P)}\to\ZZ_{K(P)}$ commutes with the canonical $T^m$-action, it induces a self-map $\underline{u}\colon M\to M$ since $M$ is the quotient of the restriction of the canonical $T^m$-action to a certain subtorus, and by construction it satisfies the commutative diagram
$$\xymatrix{T^{m-n}\ar[d]^{\underline{u}}\ar[r]&\ZZ_{K(P)}\ar[r]^\pi\ar[d]^{\underline{u}}&M\ar[d]^{\underline{u}}\\
T^{m-n}\ar[r]&\ZZ_{K(P)}\ar[r]^\pi&M}$$
where $\underline{u}\colon T^{m-n}\to T^{m-n}$ is the product of the degree $u$ map of $S^1$. Then by Proposition \ref{H^2} and naturality of transgression, we see that the self-map $\underline{u}\colon M\to M$ has the desired property.
\end{proof}

We now recall the result of \cite{MNT}, here we reproduce the proof in order to clarify naturality. Let $X$ be a CW-complex of finite type connected satisfying
\begin{enumerate}
\item $H_{\rm odd}(X;\Z)=0$ and $H_{\rm even}(X;\Z)$ is free, and 
\item there is a self-map $\varphi\colon X\to X$ satisfying $\varphi_*=u^k\colon H_{2k}(X;\Z)\to H_{2k}(X;\Z)$ for any $k\ge 0$, where $u$ is an integer whose modulo $p$ reduction is the primitive $(p-1)^\text{\rm th}$ root of unity of $\Z/p$.
\end{enumerate}
Define a self map $\alpha_i\colon\Sigma X\to\Sigma X$ by $\alpha_i:=(\Sigma\varphi-u^1)\circ\cdots\circ\widehat{(\Sigma\varphi-u^i)}\circ\cdots\circ(\Sigma\varphi-u^{p-1})$ for $i=1,\ldots,p-1$. Then $(\alpha_i)_*\colon\widetilde{H}_{2k+1}(\Sigma X;\Z/p)\to\widetilde{H}_{2k+1}(\Sigma X;\Z/p)$ is trivial for $k\not\equiv i\mod p-1$ and is the isomorphism for $k\equiv i\mod p-1$. Put 
$$X_i=\mathrm{hocolim}\{\Sigma X_{(p)}\xrightarrow{\alpha_i}\Sigma X_{(p)}\xrightarrow{\alpha_i}\Sigma X_{(p)}\xrightarrow{\alpha_i}\cdots\}.$$
Then it is easy to check that $X_i$ is $p$-locally of finite type and 
$$\widetilde{H}_{2k+1}(X_i;\Z/p)=\begin{cases}\widetilde{H}_{2k+1}(\Sigma X;\Z/p)&k\equiv i\mod p-1\\0&k\not\equiv i\mod p-1\end{cases}$$
such that the canonical map $\Sigma X_{(p)}\to X_i$ induces the projection in mod $p$ homology. Then the composite $\Sigma X_{(p)}\to\Sigma X_{(p)}\vee\cdots\vee\Sigma X_{(p)}\to X_1\vee\cdots\vee X_{p-1}$ is an isomorphism in mod $p$ homology, hence an isomorphism in homology with coefficient $\Z_{(p)}$ since spaces on both sides are $p$-locally of finite type, where the first arrow in the composite is defined by using the suspension comultiuplication. Therefore by the J.H.C. Whitehead theorem we obtain:

\begin{lemma}
[Mimura, Nishida, and Toda \cite{MNT}]
\label{MNT}
Let $X$ and $X_i$ be as above. There is a homotopy equivalence
$$\Sigma X_{(p)}\simeq X_1\vee\cdots\vee X_{p-1}$$
such that $\widetilde{H}_*(X_i;\Z/p)=0$ unless $*\equiv 2i+1\mod 2(p-1)$ for $i=1,\ldots,p-1$.
\end{lemma}

We now prove the main results.

\begin{proof}
[Proof of Theorem \ref{main-split}]
Combine Lemma \ref{power-map} and \ref{MNT}
\end{proof}

\begin{proof}
[Proof of Corollary \ref{main-rigid}]
Recall that $M$ is of dimension $2n$. Apply Theorem \ref{main-split} to $M$, then we get $\Sigma M_{(p)}\simeq X_1\vee\cdots\vee X_{p-1}$. If $p>n$, the space $X_i$ satisfies $\widetilde{H}_*(X_i;\Z/p)=0$ unless $*=2i+1$. Then since $X_i$ is simply connected, $X_i$ is a wedge of $S^{2i+1}_{(p)}$, where the number of spheres is the $2i$-dimensional Betti number of $M$ which is equal to $h_i(P)$ by Proposition \ref{h-vector}. So we obtain a homotopy equivalence $\Sigma M_{(p)}\simeq\bigvee_{i=1}^{p-1}\bigvee^{h_i(P)}S^{2i+1}_{(p)}$. We can get the same homotopy equivalence for $N$ as well, and therefore the proof is completed.
\end{proof}

\begin{proof}
[Proof of Theorem \ref{main-projection}]
Define a map $\beta_i\colon\Sigma \ZZ_{K(P)}\to\Sigma\ZZ_{K(P)}$ by $\beta_i=(\Sigma\underline{u}-u^1)\circ\cdots\circ\widehat{(\Sigma\underline{u}-u^i)}\circ\cdots\circ(\Sigma\underline{u}-u^{p-1})$ for $i=1,\ldots,p-1$, where $u$ is an integer whose modulo $p$ reduction is the primitive $(p-1)^\text{\rm th}$ root of unity of $\Z/p$. Put
$$Y_i=\mathrm{hocolim}\{\Sigma(\ZZ_{K(P)})_{(p)}\xrightarrow{\beta_i}\Sigma(\ZZ_{K(P)})_{(p)}\xrightarrow{\beta_i}\Sigma(\ZZ_{K(P)})_{(p)}\xrightarrow{\beta_i}\cdots\}.$$
By naturality of the homotopy equivalence \eqref{BBCG} with respect to self-maps of $S^1$, the self-map $\underline{u}\colon\Sigma\ZZ_{K(P)}\to\Sigma\ZZ_{K(P)}$ is identified with a wedge of the degree $u^{|I|}$ maps
$$u^{|I|}\colon\Sigma^{|I|+2}|K(P)_I|\to\Sigma^{|I|+2}|K(P)_I|$$
for $\emptyset\ne I\subset[m]$. Then we have $Y_i=\bigvee_{\substack{\emptyset\ne I\subset[m]\\|I|\equiv i\mod p-1}}\Sigma^{|I|+2}|K(P)_I|$ and the canonical map $\Sigma(\ZZ_{K(P)})_{(p)}\to Y_i$ is the projection similarly to the proof of Proposition \ref{MNT}. So the composite  $\Sigma(\ZZ_{K(P)})_{(p)}\to\Sigma(\ZZ_{K(P)})_{(p)}\vee\cdots\vee\Sigma(\ZZ_{K(P)})_{(p)}\to Y_1\vee\cdots\vee Y_{p-1}$ is a homotopy equivalence, where the first map is defined by the suspension comultiplication and the second map is a wedge of the canonical maps into the homotopy colimits. On the other hand, by Lemma \ref{power-map} there is a commutative diagram
$$\xymatrix{\Sigma\ZZ_{K(P)}\ar[r]^{\beta_i}\ar[d]^{\Sigma\pi}&\Sigma\ZZ_{K(P)}\ar[d]^{\Sigma\pi}\\
\Sigma M\ar[r]^{\alpha_i}&\Sigma M}$$
where $\alpha_i$ is as above. Then there are maps $\pi_i\colon Y_i\to X_i$ satisfying a commutative diagram
$$\xymatrix{\Sigma(\ZZ_{K(P)})_{(p)}\ar[d]^{\Sigma\pi_{(p)}}\ar[d]\ar[r]&Y_1\vee\cdots\vee Y_{p-1}\ar[d]^{\pi_1\vee\cdots\vee\pi_{p-1}}\\
\Sigma M_{(p)}\ar[r]&X_1\vee\cdots\vee X_{p-1}}$$
where the horizontal arrows are the prescribed homotopy equivalences. Thus the proof is completed.
\end{proof}

\begin{proof}
[Proof of Corollary \ref{main-projection-trivial}]
Since $p>n$, the map $\Sigma\pi_{(p)}\colon\Sigma(\ZZ_{K(P)})_{(p)}\to\Sigma M_{(p)}$ is identified with a wedge of the maps $\bigvee_{I\subset[m],\,|I|=i}(\Sigma^{|I|+2}|K(P)_I|)_{(p)}\to\bigvee S^{2i+1}_{(p)}$ for $i=1,\ldots,p-1$. If $\dim K(P)_I= |I|-1$, then $K(P)$ is a simplex, so $|K(P)_I|$ is contractible. Then $\bigvee_{I\subset[m],\,|I|=i}\Sigma^{|I|+1}|K(P)_I|$ is homotopy equivalent to a CW-complex of dimension at most $2i$, completing the proof. 
\end{proof}

We close this section by showing a general homotopy theoretical property of finite complexes consisting only of even cells from which Corollary \ref{main-rigid} also follows since there are cell decompositions of quasitoric manifolds only by even dimensional cells.

\begin{proposition}
\label{split-even}
Let $X$ be an $n$-dimensional connected finite complex consisting only of even cells. If $p>n$, then $\Sigma X_{(p)}$ is homotopy equivalent to a wedge of $p$-local odd spheres.
\end{proposition}

\begin{proof}
Induct on the skeleta of $X$. We may assume the 0-skeleton is a point since $X$ is connected, so the claim is trivially true for the 0-skeleton. Suppose that $\Sigma X^{(2k-2)}_{(p)}\simeq\bigvee_{i=1}^{k-1}\bigvee^{m_i}S^{2i+1}_{(p)}$. Then the attaching maps of  $(2k+1)$-cells of $\Sigma X_{(p)}$ are identified with maps $S^{2k}\to\bigvee_{i=1}^{k-1}\bigvee^{m_i}S^{2i+1}_{(p)}$. By the Hilton-Milnor theorem, $\Omega(\bigvee_{i=1}^{k-1}\bigvee^{m_i}S^{2i+1}_{(p)})$ is homotopy equivalent to a weak product of the loop spaces of $p$-local odd spheres of dimension $\ge 3$. Then since $p>k$ and $\pi_{2j}(S^{2\ell+1})_{(p)}=0$ for $j<\ell+p-1$, the attaching maps are null homotopic, hence the induction proceeds.
\end{proof}

\section{Non-triviality of the projection $\pi$}

Let $M$ be a quasitoric manifold over an $n$-polytope $P$ and let $\pi\colon\ZZ_{K(P)}\to M$ denote the projection. By Corollary \ref{main-projection-trivial}, $\Sigma \pi_{(p)}$ is trivial for $p>n$. So one would ask whether $\pi_{(p)}$ and $\Sigma^\infty\pi$ are trivial or not. This section shows non-triviality of $\pi_{(p)}$ and examines non-triviality of $\Sigma^\infty\pi$ for quasitoric manifolds over a product of simplices and low dimensional quasitoric manifolds. We first consider the $p$-localization.

\begin{proposition}
The $p$-localization $\pi_{(p)}$ is not null homotopic.
\end{proposition}

\begin{proof}
Recall that there is a homotopy fibration \eqref{fibration}. Then if $\pi_{(p)}$ were null homotopic, we would have $T^{m-n}_{(p)}\simeq(\ZZ_{K(P)})_{(p)}\times\Omega M_{(p)}$. It is shown in \cite{MN} that if $X$ is a simply connected $p$-local finite complex which is not contractible, then $\pi_*(X)$ has torsion. By \eqref{fibration} we also see that $\ZZ_{K(P)}$ is not contractible at any prime $p$, since $M_{(p)}$ is $p$-locally finite but $BT^{m-n}_{(p)}$ is not. Then since $\ZZ_{K(P)}$ is simply connected, we get that $T^{m-n}_{(p)}\simeq(\ZZ_{K(P)})_{(p)}\times\Omega M_{(p)}$ has torsion in homotopy groups, a contradiction.
\end{proof}

We next consider non-triviality of $\Sigma^\infty\pi$ for quasitoric manifolds over a product of simplices. We start with the easiest case. Recall that the complex projective space $\C P^n$ is the only quasitoric manifold over the $n$-simplex $\Delta^n$, and that the projection $\pi$ is the canonical map $S^{2n+1}\to\C P^n$. Then since the cofiber of $\pi$ is $\C P^{n+1}$ whose top cell does not split after stabilization, one sees that $\Sigma^\infty\pi$ is not null homotopic. We here record this almost trivial fact.

\begin{lemma}
\label{CP^k}
The projection $\pi\colon\ZZ_{K(\Delta^n)}\to\C P^n$ is not null homotopic after stabilization.
\end{lemma}

It is helpful to recall the fact on moment-angle complexes regarding products of simple polytopes. For simple polytopes $P_1,P_2$ the product $P_1\times P_2$ is also a simple polytope and $K(P_1\times P_2)=K(P_1)*K(P_2)$, the join of $K(P_1)$ and $K(P_2)$. By definition we have $\ZZ_{K(P_1\times P_2)}=\ZZ_{K(P_1)*K(P_2)}=\ZZ_{K(P_1)}\times\ZZ_{K(P_2)}$, and in particular $\ZZ_{K(P_1)}$ is a retract of $\ZZ_{K(P_1\times P_2)}$. We prepare a simple lemma.

\begin{lemma}
\label{criterion-proj}
Let $P$ be a simple polytope, and let $M$ be a quasitoric manifold over $P\times\Delta^k$. If there is a map $q\colon M\to\C P^k$ satisfying a homotopy commutative diagram
$$\xymatrix{\ZZ_{K(P\times\Delta^k)}\ar[r]^{\rm proj}\ar[d]^\pi&\ZZ_{K(\Delta^k)}\ar[d]^\pi\\
M\ar[r]^q&\C P^k,}$$
then the projection $\pi\colon\ZZ_{K(P\times\Delta^k)}\to M$ is not null homotopic after stabilization.
\end{lemma}

\begin{proof}
Since $\ZZ_{K(\Delta^k)}$ is a retract of $\ZZ_{K(P\times\Delta^k)}$, it follows from Lemma \ref{CP^k} that $\Sigma^\infty(q\circ\pi)$ is not null homotopic. Therefore since $\Sigma^\infty(q\circ\pi)=\Sigma^\infty q\circ\Sigma^\infty\pi$, the proof is completed.
\end{proof}

There is a class of generic quasitoric manifolds over a product of simplices called generalized Bott manifolds which have been intensively studied in toric topology. See \cite{CMS} for details. 
By definition a generalized Bott manifold $B$ over $\Delta^{n_1}\times\cdots\times\Delta^{n_\ell}$ satisfies a commutative diagram
$$\xymatrix{\ZZ_{K(\Delta^{n_1}\times\cdots\times\Delta^{n_\ell})}\ar[r]\ar[d]^\pi&\ZZ_{K(\Delta^{n_1}\times\cdots\times\Delta^{n_{\ell-1}})}\ar[r]\ar[d]^\pi&\cdots\ar[r]&\ZZ_{K(\Delta^{n_1}\times\Delta^{n_2})}\ar[r]\ar[d]^\pi&\ZZ_{K(\Delta^{n_1})}\ar[d]^\pi\\
B_\ell\ar[r]^{q_\ell}&B_{\ell-1}\ar[r]^{q_{\ell-1}}&\cdots\ar[r]^{q_2}&B_2\ar[r]^{q_1}&B_1}$$
where the upper horizontal arrows are the projections. Since $B_1=\C P^{n_1}$, we get the following by Lemma \ref{criterion-proj}

\begin{corollary}
\label{Bott}
If $B$ is a generalized Bott manifold over $\Delta^{n_1}\times\cdots\times\Delta^{n_\ell}$, then the projection $\pi\colon\ZZ_{K(\Delta^{n_1}\times\cdots\times\Delta^{n_\ell})}\to B$ is not null homotopic after stabilization.
\end{corollary}

In order to examine non-triviality of $\Sigma^\infty\pi$ for quasitoric manifolds other than Bott manifolds, we give a cohomological generalization of Lemma \ref{criterion-proj}. 

\begin{lemma}
\label{criterion-1}
Let $X$ be a space such that $H^2(X;\Z)=\Z\langle x_1,\ldots,x_k\rangle$ and $H^*(X;\Z/p)$ is generated by the mod $p$ reduction of $x_1,\ldots,x_k$ as a ring, and let $F$ be the homotopy fiber of a map $\alpha=(x_1,\ldots,x_k)\colon X\to BT^k$. Suppose the following conditions hold:
\begin{enumerate}
\item There are $x\in H^{2\ell-2i}(BT^k;\Z/p)$ and transgressive $a\in H^{2\ell-1}(F;\Z/p)$ such that
$$\tau(a)=\theta(x)$$
for some degree $2i$ Steenrod operation $\theta$.
\item There is a map $f\colon S^{2\ell-1}\to F$ such that $f^*(a)\ne 0$ in mod $p$ cohomology.
\end{enumerate}
Then the stabilization of the fiber inclusion $F\to X$ is not null homotopic.
\end{lemma}

\begin{proof}
Let $i\colon F\to X$ and $j\colon X\to C_{i\circ f}$ denote the inclusions. Then there is a commutative diagram of exact sequences
$$\xymatrix{0\ar[r]&H^{2\ell-1}(S^{2\ell-1};\Z/p)\ar[r]^\delta&H^{2\ell}(C_{i\circ f};\Z/p)\ar[r]^{j^*}&H^{2\ell}(X;\Z/p)\\
&H^{2\ell-1}(F;\Z/p)\ar[r]^\delta\ar[u]_{f^*}&H^{2\ell}(X,F;\Z/p)\ar[u]_{f^*}&H^{2\ell}(BT^k;\Z/p)\ar[l]_{\alpha^*}\ar[u]_{\alpha^*}.}$$
Put $\bar{x}:=f^*\circ\alpha^*(x)$. Since $\tau(a)=\theta(x)$, we have $\theta(\bar{x})=f^*\circ\alpha^*(\theta(x))=\delta\circ f^*(a)\ne 0$, where $\theta(x)=\tau(a)\ne 0$ since $H^{\rm odd}(X;\Z/2)=0$. Then we obtain
$$H^*(C_{i\circ f};\Z/p)\cong A\oplus\langle\theta(\bar{x})\rangle,\quad A\cong H^*(X;\Z/p)$$
as modules, implying that $\theta(A)\not\subset A$ by $\bar{x}\in A$. If $\Sigma^\infty i$ were null homotopic, we would have $\theta(A)\subset A$ which contradicts to the above calculation, so $\Sigma^\infty i$ is not null homotopic. 
\end{proof}

We apply Lemma \ref{criterion-1} to quasitoric manifolds over a product of two simplices which are not necessarily generalized Bott manifolds.

\begin{proposition}
\label{two}
If $M$ is a quasitoric manifold over $\Delta^k\times\Delta^{n-k}$ and neither $n+2$ nor $n-k+2$ is a power of $2$, then $\Sigma^\infty\pi$ is not null homotopic.
\end{proposition}

\begin{proof}
By Proposition \ref{CP^k} we may assume $0<k<n$. By \cite{CMS} the mod 2 cohomology of $M$ is given by
$$H^*(M;\Z/2)=\Z/2[x,y]/(x^{k'-\ell+1}(x+y)^\ell,y^{n-k'+1})$$
for some $\ell\ge 0$, where $k'=k$ or $k'=n-k$.
It is sufficient to check that $M$ satisfies the conditions of Lemma \ref{criterion-1}.
By Proposition \ref{H^2} $M$ satisfies the conditions for the space $X$ of Lemma \ref{criterion-1}.
It follows immediately from Lucas' theorem that $\theta(t^r)=t^{n-k'+1}$ for some Steenrod operation $\theta$ and $t\in H^2(BT^2;\Z/2)$ corresponding to $y$, so the latter part of the condition (1) is satisfied.
Since $\ZZ_{K(\Delta^k\times\Delta^{n-k})}=S^{2k+1}\times S^{2(n-k)+1}$, there is $a\in H^{2(n-k')+1}(\ZZ_{K(\Delta^k\times\Delta^{n-k})};\Z/2)$ satisfying $\tau(a)=t^{n-k'+1}$ for a degree reason.
Then the former part of the condition (1) is also satisfied. Moreover, the element $a$ is spherical, so the condition (2) is satisfied. 
\end{proof}

We next specialize Lemma \ref{criterion-1} for applications to low dimensional quasitoric manifolds.



\begin{proposition}
\label{criterion-x^2}
Let $M$ be a quasitoric manifold. If there is non-zero $x\in H^2(M;\Z/2)$ satisfying $x^2=0$, then $\Sigma^\infty\pi$ is not null homotopic. 
\end{proposition}

\begin{proof}
It is sufficient to check that the conditions of Lemma \ref{criterion-1} are satisfied. Let  $P$ be a polytope on which $M$ stands. Since $\ZZ_{K(P)}$ is 2-connected, there is $a\in H^3(\ZZ_{K(P)};\Z/2)$ satisfying $\tau(a)=t^2$, where $t\in H^2(BT^{m-n};\Z/2)$ satisfies $\alpha^*(t)=x$. Then for $t^2=\mathrm{Sq}^2t$, the condition (1) of Lemma \ref{criterion-1} is satisfied. We also have that the Hurewicz map $\pi_3(\ZZ_{K(P)})\to H_3(\ZZ_{K(P)};\Z)$ is an isomorphism, so any element of $H^3(\ZZ_{K(P)};\Z/2)$ is spherical. Then the condition (2) of Lemma \ref{criterion-1} is satisfied, and therefore the proof is done. 
\end{proof}

We now apply Proposition \ref{criterion-x^2} to low dimensional quasitoric manifolds.

\begin{corollary}
\label{h}
If $M$ is a 4-dimensional quasitoric manifold, then $\Sigma^\infty\pi$ is not null homotopic.
\end{corollary}

\begin{proof}
Suppose that the quasitoric manifold $M$ stands over a 2-polytope $P$. If $P=\Delta^2$, the corollary follows from Lemma \ref{CP^k} since $\C P^2$ is the only quasitoric manifold over $\Delta^2$. If $P\ne\Delta^2$, then $P$ is a $k$-gon for $k\ge 4$, hence $h_2(P)=1<k-2=h_1(P)$. Then it follows from  Proposition \ref{h-vector} that $\dim H^4(M;\Z/2)<\dim H^2(M;\Z/2)$, implying that there must be non-zero $x\in H^2(M;\Z/2)$ satisfying $x^2=0$. Thus the proof is completed by Proposition \ref{criterion-x^2}.
\end{proof}

\begin{remark}
We here remark that $h_1(P)=h_2(P)$ by the Dehn-Sommerville equation for $\dim P=3$ and $h_1(P)>h_2(P)$ for $\dim P>3$ by the $g$-theorem (cf. \cite{BP}), so the argument in the proof of Corollary \ref{h} does not work for $\dim P\ge 3$.
\end{remark}

\begin{corollary}
\label{cube}
If $M$ is a quasitoric manifold over the 3-cube, then $\Sigma^\infty\pi$ is not null homotopic.
\end{corollary}

\begin{proof}
It is calculated in \cite{CMS,H} that the mod 2 cohomology of $M$ is given by
$$H^*(M;\Z/2)=\Z/2[x,y,z]/(x^2+x(ay+bz),y^2+y(cx+dz),z^2+z(ex+fy))$$
for $a,b,c,d,e,f\in\Z/2$ satisfying
$$ac=df=0,\quad\begin{vmatrix}1&c&e\\a&1&f\\b&d&1\end{vmatrix}=1.$$
We now suppose that $w^2\ne 0$ for all non-zero $w\in H^2(M;\Z/2)$. Then for $x_1^2\ne 0$ we have $(a,b)$ is either $(1,0),(0,1),(1,1)$. Consider the case $(a,b)=(1,0)$. That $a=1$ implies $c=0$, so $d=1$ since $y^2\ne 0$. Then $f=0$, implying $e=1$ since $z^2\ne 0$. Hence we obtain $\begin{vmatrix}1&c&e\\a&1&f\\b&d&1\end{vmatrix}=\begin{vmatrix}1&0&1\\1&1&0\\0&1&1\end{vmatrix}=0$, a contradiction. In the case $(a,b)=(0,1),(1,1)$ we can similarly get $(c,d,e,f)=(0,1,1,0)$, so a contradiction occurs. Thus there is non-zero $w\in H^2(M;\Z/2)$ with $w^2=0$, and therefore the proof is done by Proposition \ref{criterion-x^2}.
\end{proof}

For the last we dare to conjecture the following from Proposition \ref{criterion-x^2} and Corollary \ref{Bott}, \ref{two}, \ref{h}, \ref{cube}.

\begin{conjecture}
For any quasitoric manifold $M$, $\Sigma^\infty\pi$ is not null homotopic.
\end{conjecture}

\end{document}